\documentclass[12pt]{amsart}
\usepackage{amsfonts}
\usepackage{amssymb}
\usepackage{amsmath}
\usepackage{amsthm}

\theoremstyle{plain}
\newtheorem{thm}{Theorem}[section]
\newtheorem{lem}[thm]{Lemma}
\theoremstyle{definition}
\newtheorem{defn}[thm]{Definition}
\newtheorem{ex}[thm]{Example}
\newtheorem{rem}[thm]{Remark}

\numberwithin{equation}{section}

\begin{document}

\title[Systems with coupled nonlinear BCs]{Multiple positive solutions of systems with
coupled nonlinear BCs}
\author{Gennaro Infante}%
\address{Gennaro Infante, Dipartimento di Matematica ed Informatica, Universit\`{a} della
Calabria, 87036 Arcavacata di Rende, Cosenza, Italy}
\email{gennaro.infante@unical.it}
\author{Paolamaria Pietramala}%
\address{Paolamaria Pietramala, Dipartimento di Matematica ed Informatica, Universit\`{a}
della Calabria, 87036 Arcavacata di Rende, Cosenza, Italy}
\email{pietramala@unical.it}
\begin{abstract} 
Using the theory of fixed point index, we discuss the existence and multiplicity of non-negative solutions of a wide class of boundary value problems with coupled nonlinear boundary conditions. Our approach is fairly general and covers a variety of situations.
We illustrate our theory in an example all the constants that occur in our theory.
\end{abstract}
\maketitle
\section{Introduction}
The aim of this paper is to present a theory for the existence of positive solution for a fairly general class of systems of ordinary differential equations subject to nonlinear, nonlocal boundary conditions. In particular we are interested in systems that present a coupling in the boundary conditions (BCs);
this type of problems have been studied in \cite{Amann, Asif-khan-jmma, cui-sun, Goodrich1, Goodrich2, Leung, meh-nic, yjoa} and often occur in applications, for example when modelling the displacement of a suspension bridge subject to nonlinear controllers. 

In \cite{Lu-Yu-Liu}, Lu and co-authors, by means of the Krasnosel'ski\u\i{}-Guo Theorem on cone compressions and cone expansions, studied existence of positive solutions of the system of ordinary differential equations (ODEs)
\begin{equation}  \label{lyl}
\begin{array}{c}
u^{\prime \prime }(t)+f_{1}(t, v(t))=0,\ t \in (0,1), \\
v^{(4)}(t)=f_{2}(t, u(t)),\ t \in (0,1),%
\end{array}%
\end{equation}
subject to the BCs
\begin{equation}  \label{bc-hom1}
\begin{array}{c}
u(0)= u(1)=v(0)=v(1)=v^{\prime \prime }(0)=v^{\prime \prime
}(1)=0.%
\end{array}%
\end{equation}

The motivation, given in \cite{Lu-Yu-Liu}, for studying the BVP \eqref{lyl}-\eqref{bc-hom1} is that it can be seen as the stationary case of a model for the oscillations of the center-span of a suspension bridge, where the forth order equation represents the road-bed (seen as an elastic beam) and second order equation models the main cable (seen as a vibrating string). The BCs in this case illustrate the fact that the beam is simply supported and that the two ends of the cable are supposed to be immovable, see also, for example, \cite{Lazer-McKenna, Matas}. 

The existence of positive solutions of a coupled system with an elastic beam equation of the type
\begin{equation}  \label{sun-eq}
\begin{array}{c}
u^{\prime \prime }(t)+f_{1}(t, v(t))=0,\ t \in (0,1), \\
v^{(4)}(t)=f_{2}(t, u(t), v(t)),\ t \in (0,1),%
\end{array}%
\end{equation}
has been studied by Sun in \cite{sun}, by monotone iterative techniques, under the BCs
\begin{equation}  \label{bc-hom2}
\begin{array}{c}
u(0)= u(1)=v(0)=v(1)=v^{\prime }(0)=v^{\prime \prime}(1)=0.%
\end{array}%
\end{equation}

A common feature of the systems \eqref{sun-eq}-\eqref{bc-hom2} and \eqref{lyl}-\eqref{bc-hom1} is that the BCs under consideration are local and homogeneous.

In \cite{gifmpp-cnsns}, Infante and co-authors, by means of classical fixed point index theory, provided a fairly general theory suitable to study the existence of non-negative solutions of a variety of systems of ODEs subject to \emph{linear}, \emph{nonlocal} conditions, one example being the system
\begin{equation}  \label{eq1}
\begin{array}{c}
u^{\prime \prime }(t)+g_{1}(t)f_{1}(t,u(t),v(t))=0,\ t \in (0,1), \\
v^{(4)}(t)=g_{2}(t)f_{2}(t,u(t),v(t)),\ t \in (0,1),%
\end{array}%
\end{equation}
with the BCs
\begin{equation}  \label{bc1}
u(0)=\beta_{11}[u], u(1)=\delta_{12}[v], 
v(0)=\beta_{21}[v], v^{\prime \prime }(0)=0, v(1)=0,
 v^{\prime \prime}(1)+\delta_{22}[u]=0,%
\end{equation}
where
$\beta _{ij}[\cdot]$, $\delta _{ij}[\cdot]$ are bounded linear functionals given by Riemann-Stieltjes integrals, namely
\begin{equation*}
\beta _{ij}[w]=\int_{0}^{1}w(s)\,dB_{ij}(s),\,\,\,\,\delta _{ij}[w]=\int_{0}^{1}w(s)\,dC_{ij}(s).
\end{equation*}%
This type of formulation includes, as special cases, multi-point or integral conditions, when
\begin{equation*}
\alpha_{ij}[w]=\sum_{l=1}^{m} \alpha_{ijl}w(\eta_{ijl})\ \text{and}\
\alpha_{ij}[w]=\int_{0}^{1} {\alpha}_{ij}(s)w(s)\,ds,
\end{equation*}
see for example~\cite{hen-luca, hen-nto-pur, Jank1, Jank2, kttmna, kong1, rma, sotiris, jw-jlms, jw-gi-jlmsII}.

In the case of the system \eqref{eq1}-\eqref{bc1}, the BCs
\begin{equation*}
u(0)= \beta u(\xi), u(1)=v(1)=v^{\prime \prime}(0)= v(0)=0, v^{\prime \prime }(1)+\delta
u(\eta)=0,
\end{equation*}
can be interpreted as a cable-beam model with two devices of feedback control, where the displacement of the left end of cable is related to displacement of another point $\xi$ of the cable and the bending moment in the right end of the beam depends upon the displacement registered in a point $\eta$
of the string. We point out that not necessarily the response of the controllers needs to be of linear type, for example this happens with conditions of the type
\begin{equation*}
u(0)= H (u(\xi)), u(1)=v(1)=v^{\prime \prime }(0)= v(0)=0, v^{\prime \prime }(1)+L(u(\eta))=0;
\end{equation*}
we refer to \cite{gipp-cant} for more details regarding the illustration of nonlinear controllers on a beam. 

Our approach allows us to deal with a larger class of \emph{nonlinear} nonlocal BCs, one example given by the BCs
\begin{gather}
\begin{aligned}  \label{bc1-gen}
u(0)=H_{11}(\beta_{11}[u])+L_{11}(\delta_{11}[v]),\ u(1)&=H_{12}(\beta_{12}[u])+L_{12}(\delta_{12}[v]), \\
v(0)=H_{21}(\beta_{21}[v])+L_{21}(\delta_{21}[u]),&\ v^{\prime \prime }(0)=0,\ v(1)=0,\\ 
 \ v^{\prime \prime
}(1)+H_{22}(\beta_{22}[v])+L_{22}&(\delta_{22}[u])=0,%
\end{aligned}
\end{gather}
where $H_{ij}$, $L_{ij}$ are continuous functions. For earlier contributions on problems with nonlinear BCs we refer the reader to~\cite{amp, Cabada1, acfm-nonlin, dfdorjp, Goodrich3, Goodrich4, Goodrich5, gi-caa, gipp-cant, paola} and references therein.

Here we develop an existence theory for multiple positive solutions of the perturbed Hammerstein integral
equations of the type
\begin{gather*}
\begin{aligned}
u(t)=\sum_{i=1,2}\gamma_{1i}(t)\Bigl(H_{1i}(\beta_{1i} [u])+L_{1i}(\delta_{1i}
[v])\Bigr)+\int_{0}^{1}k_1(t,s)g_1(s)f_1(s,u(s),v(s))\,ds, \\
v(t)= \sum_{i=1,2}\gamma_{2i}(t)\Bigl(L_{2i}(\delta_{2i}
[u])+H_{2i}(\beta_{2i} [v])\Bigr)+\int_{0}^{1}k_2(t,s)g_2(s)f_2(s,u(s),v(s))\,ds.%
\end{aligned}
\end{gather*}
Similar systems of perturbed Hammerstein integral equations were studied in \cite{df-gi-do, Goodrich1, Goodrich2,gipp-ns, gipp-nonlin, gifmpp-cnsns, kang-wei, ya1}.
Our theory covers,  as a special case, the system \eqref{eq1}-\eqref{bc1-gen} and we show in an example that all the constants that occur in our theory can be computed.

We make use of the classical fixed point index theory (see
for example \cite{Amann-rev, guolak}) and also benefit of ideas from the papers
\cite{gi-caa, gipp-nonlin, gifmpp-cnsns, gi-jw-ems,  jw-gi-jlms}.

\section{Positive solutions for systems of perturbed integral equations}
We begin with stating some assumptions on the terms that occur in the system of perturbed Hammerstein integral equations
\begin{gather}
\begin{aligned}\label{syst}
u(t)=\sum_{i=1,2}\gamma_{1i}(t)\Bigl(H_{1i}(\beta_{1i} [u])+L_{1i}(\delta_{1i}
[v])\Bigr)+F_1(u,v)(t), \\
v(t)= \sum_{i=1,2}\gamma_{2i}(t)\Bigl(L_{2i}(\delta_{2i}
[u])+H_{2i}(\beta_{2i} [v])\Bigr)+F_2(u,v)(t),%
\end{aligned}
\end{gather}
where
\begin{equation*}
F_i(u,v)(t):=\int_{0}^{1}k_i(t,s)g_i(s)f_i(s,u(s),v(s))\,ds,
\end{equation*}
namely:
\begin{itemize}
\item For every $i=1,2$, $f_i: [0,1]\times [0,\infty)\times [0,\infty) \to
[0,\infty)$ satisfies Carath\'{e}odory conditions, that is, $f_i(\cdot,u,v)$
is measurable for each fixed $(u,v)$ and $f_i(t,\cdot,\cdot)$ is continuous
for almost every (a.e.) $t\in [0,1]$, and for each $r>0$ there exists $\phi_{i,r} \in
L^{\infty}[0,1]$ such that{}
\begin{equation*}
f_i(t,u,v)\le \phi_{i,r}(t) \;\text{ for } \; u,v\in [0,r]\;\text{ and
a.\,e.} \; t\in [0,1].
\end{equation*}%
{}
\item For every $i=1,2$, $k_i:[0,1]\times [0,1]\to [0,\infty)$ is
measurable, and for every $\tau\in [0,1]$ we have
\begin{equation*}
\lim_{t \to \tau} |k_i(t,s)-k_i(\tau,s)|=0 \;\text{ for a.\,e.}\, s \in [0,1].
\end{equation*}%
{}
\item For every $i=1,2$, there exist a subinterval $[a_i,b_i] \subseteq
[0,1] $, a function $\Phi_i \in L^{\infty}[0,1]$, and a constant $c_{i} \in
(0,1]$, such that
\begin{align*}
k_i(t,s)\leq \Phi_i(s) \text{ for } &t \in [0,1] \text{ and a.\,e.}\, s\in [0,1], \\
k_i(t,s) \geq c_{i}\Phi_i(s) \text{ for } &t\in [a_i,b_i] \text{ and a.\,e.} \, s \in [0,1].
\end{align*}%
{}
\item For every $i=1,2$, $g_i\,\Phi_i \in L^1[0,1]$, $g_i \geq 0$ a.e., and $%
\int_{a_i}^{b_i} \Phi_i(s)g_i(s)\,ds >0$.
{}
\item For every $i,j=1,2$, $\beta _{ij}[\cdot ]$ and $\delta _{ij}[\cdot ]$ are linear functionals
given by
\begin{equation*}
\beta _{ij}[w]=\int_{0}^{1}w(s)\,dB_{ij}(s),\,\,\,\,\delta _{ij}[w]=\int_{0}^{1}w(s)\,dC_{ij}(s),
\end{equation*}%
involving Riemann-Stieltjes integrals; $B_{ij}$ and $C_{ij}$ are of bounded variation and $%
dB_{ij},dC_{ij}$ are \emph{positive} measure.
\item  $H_{ij},L_{ij}: [0,\infty)\to [0,\infty)$ are continuous functions such that
there exist
$h_{ij1}, h_{ij2}, l_{ij2} \in [0,\infty)$, $i,j=1,2$,
with
\begin{equation*}
 h_{ij1}w \leq H_{ij}(w)\leq  h_{ij2}w,\,\,\,\,  L_{ij}(w)\leq  l_{ij2}w,
\end{equation*}
for every $w\geq 0$.
\item $\gamma_{ij} \in C[0,1], \;\gamma_{ij}(t) \geq 0\; \;\text{for every}\;t\in [0,1],\;\;
 h_{ij2}\beta_{ij}[\gamma_{ij}] <1$ and
there exists $c_{ij} \in(0,1]$ such that
\begin{equation*}
\gamma _{ij}(t)\geq c_{ij}\| \gamma _{ij}\| _{\infty }\;\text{for every%
}\;t\in \lbrack a_{i},b_{i}],
\end{equation*}%
where $\| w\| _{\infty }:=\max \{|w(t)|,\;t\;\in \lbrack 0,1]\}$.
\item  $D_i:=(1-h_{i12}\beta_{i1}[\gamma_{i1}])(1-h_{i22}\beta_{i2}[\gamma_{i2}])
-h_{i12}h_{i22}\beta_{i1}[\gamma_{i2}]\beta_{i2}[\gamma_{i1}]> 0,\; i=1,2$.{}
\end{itemize}
It follows from $D_i>0$ that
  $$\underline{D}_i:=(1-h_{i11}\beta_{i1}[\gamma_{i1}])(1-h_{i21}\beta_{i2}[\gamma_{i2}])
-h_{i11}h_{i21}\beta_{i1}[\gamma_{i2}]\beta_{i2}[\gamma_{i1}]> 0.$$

We work in the space $C[0,1]\times C[0,1]$ endowed with the norm
\begin{equation*}
\| (u,v)\| :=\max \{\| u\| _{\infty },\| v\| _{\infty }\}.
\end{equation*}%
Let
\begin{equation*}
\tilde{K_{i}}:=\{w\in C[0,1]:w(t)\geq 0\ \text{for}\ t\in \lbrack 0,1]\,\,%
\text{and}\,\,\min_{t\in \lbrack a_{i},b_{i}]}w(t)\geq \tilde{c_{i}}\|
w\| _{\infty }\},
\end{equation*}%
where $\tilde{c_{i}}=\min \{c_{i},c_{i1},c_{i2}\}$, and consider the cone $K$
in $C[0,1]\times C[0,1]$ defined by
\begin{equation*}
\begin{array}{c}
K:=\{(u,v)\in \tilde{K_{1}}\times \tilde{K_{2}}\}.%
\end{array}%
\end{equation*}
For a \emph{positive} solution of the system \eqref{syst} we mean a solution
$(u,v)\in K$ of \eqref{syst} such that $\|(u,v)\|>0$.

Under our assumptions, it is routine to show that the integral operator
\begin{gather*}
\begin{aligned}    
T(u,v)(t):=& 
\left(
\begin{array}{c}
\sum_{i=1,2}\gamma_{1i}(t)\Bigl(H_{1i}(\beta_{1i} [u])+L_{1i}(\delta_{1i}
[v])\Bigr)+F_1(u,v)(t) \\
\sum_{i=1,2}\gamma_{2i}(t)\Bigl(L_{2i}(\delta_{2i}
[u])+H_{2i}(\beta_{2i} [v])\Bigr)+F_2(u,v)(t)%
\end{array}
\right)
\\ &:=
\left(
\begin{array}{c}
T_1(u,v)(t) \\
T_2(u,v)(t)%
\end{array}
\right) ,
\end{aligned}
\end{gather*}
leaves the cone $K$ invariant and is compact, see for example Lemma 1 of~\cite{gifmpp-cnsns}.

We use the following (relative) open bounded sets in $K$:
\begin{equation*}
K_{\rho} = \{ (u,v) \in K : \|(u,v)\|< \rho \},
\end{equation*}
and
\begin{equation*}
V_\rho=\{(u,v) \in K: \min_{t\in [a_1,b_1]}u(t)<\rho\ \text{and}\ \min_{t\in
[a_2,b_2]}v(t)<\rho\}.
\end{equation*}
The set $V_\rho$ (in the context of systems) was introduced by the authors in~\cite{gipp-ns} and is equal to the set
called $\Omega^{\rho /c}$ in~\cite{df-gi-do}. $\Omega^{\rho /c}$ is an extension to the case of systems of a set given by Lan~\cite{lan}. 
For our index calculations we make use of the fact that
$$K_{\rho}\subset V_{\rho}\subset K_{\rho/c},$$ 
where $c=\min\{\tilde{c_1},\tilde{c_2}\}$. We denote by $\partial K_{\rho}$ and $\partial
V_{\rho}$ the boundary of $K_{\rho}$ and $V_{\rho}$ relative to $K$.\\

We utilize the following results of~\cite{jw-gi-jlms} regarding order preserving matrices:

\begin{defn}
A $\;2 \times 2$ matrix $\mathcal{Q}$ is said to be order
preserving (or positive) if $p_{1}\geq p_{0}$, $q_{1}\geq q_{0}$
imply
\begin{equation*}
\mathcal{Q}
\begin{pmatrix}
  p_{1} \\
  q_{1}
\end{pmatrix}%
\geq \mathcal{Q}
\begin{pmatrix}
  p_{0} \\
  q_{0}
\end{pmatrix},
\end{equation*}
in the sense of components.
\end{defn}

\begin{lem}\label{lematrix2}\cite{jw-gi-jlms}
Let
\begin{equation*}
\mathcal{Q}=
\begin{pmatrix}
  a & -b \\
  -c & d
\end{pmatrix}
\end{equation*}
with $a,b,c,d\geq 0$ and $\det \mathcal{Q}> 0$. Then
$\mathcal{Q}^{-1}$ is order preserving.
\end{lem}

\begin{rem} \label{rem1}
It is a consequence of Lemma \ref{lematrix2} that if
 \begin{equation*}
\mathcal{N}=
\begin{pmatrix}
  1-a & -b\\
  -c & 1-d
\end{pmatrix},
\end{equation*}
 satisfies the hypotheses of Lemma \ref{lematrix2},
$p \geq 0, q \geq 0$ and $\mu>1$ then
$$\mathcal{N}_\mu^{-1}\begin{pmatrix}
  p \\
  q
\end{pmatrix}\leq \mathcal{N}^{-1}\begin{pmatrix}
  p \\
  q
\end{pmatrix},$$
where
 \begin{equation*}
\mathcal{N}_\mu=
\begin{pmatrix}
  \mu-a & -b\\
  -c & \mu-d
\end{pmatrix}.
\end{equation*}
\end{rem}
In the sequel of the paper we use the following notation.
\begin{align*}
\mathcal{K}_{ij}(s):=&\int_0^1 k_i(t,s) \,dB_{ij}(t),\,\,Q_i=\sum_{l=1,2}\beta_{i1}[\gamma_{il}]l_{il2}\delta_{il}[1],\,\,
  S_i=  \sum_{l=1,2}\beta_{i2}[\gamma_{il}]l_{il2}\delta_{il}[1],\\
\theta_{i1}=&\frac{1-h_{i22}\beta_{i2}[\gamma_{i2}]}{D_i},\,\,
  \theta_{i2}=\frac{ h_{i22}\beta_{i1}[\gamma_{i2}]}{D_i} ,\,\, \theta_{i3}=\frac{h_{i12}\beta_{i2}[\gamma_{i1}]}{D_i},\,\,
 \theta_{i4}=\frac{1-h_{i12}\beta_{i1}[\gamma_{i1}]}{D_i} ,
\end{align*}

We are now able to prove a result concerning the fixed point index on the set $K_{\rho}$.

\begin{lem}

\begin{enumerate}
\item[$(\mathrm{I}_{\protect\rho }^{1})$]
there exists $\rho >0$
such that for every $i=1,2$
\begin{multline}\label{eqmestt}
 f_i^{0,\rho} \Bigl( \Bigl(\| \gamma _{i1}\| _{\infty }h_{i12}\theta_{i1}+
\| \gamma _{i2}\| _{\infty }h_{i22}\theta_{i3} \Bigr)\int_0^1 \mathcal{K}_{i1}(s)g_i(s)\,ds \\
+ \Bigl(\| \gamma _{i1}\| _{\infty }h_{i12}\theta_{i2}+
\| \gamma _{i2}\| _{\infty }h_{i22}\theta_{i4} \Bigr)\int_0^1 \mathcal{K}_{i2}(s)g_i(s)\,ds +\dfrac{1}{m_i}\Bigr)\\
+\| \gamma _{i1}\| _{\infty }h_{i12}(\theta_{i1} Q_i+\theta_{i2} S_i)+
\| \gamma _{i2}\| _{\infty }h_{i22}(\theta_{i3} Q_i+\theta_{i4} S_i)+\sum_{j=1,2}\| \gamma _{ij}\| _{\infty }l_{ij2}\delta_{ij}
[1] <1 
\end{multline}{}
where
\begin{equation*}
f_{i}^{0,{\rho }}=\sup \Bigl\{\frac{f_{i}(t,u,v)}{\rho }:\;(t,u,v)\in
\lbrack 0,1]\times \lbrack 0,\rho ]\times \lbrack 0,\rho ]\Bigr\}\ \text{and}%
\ \frac{1}{m_{i}}=\sup_{t\in \lbrack 0,1]}\int_{0}^{1}k_{i}(t,s)g_{i}(s)\,ds.
\end{equation*}%
{}
\end{enumerate}

Then the fixed point index, $i_{K}(T,K_{\rho})$, is equal to 1.
\end{lem}

\begin{proof}
We show that $\mu (u,v)\neq T(u,v)$ for every $(u,v)\in \partial K_{\rho }$
and for every $\mu \geq 1$; this ensures that the index is 1 on $K_{\rho }$.
In fact, if this does not happen, there exist $\mu \geq 1$ and $(u,v)\in
\partial K_{\rho }$ such that $\mu (u,v)=T(u,v)$. Assume, without loss of
generality, that $\| u\| _{\infty }=\rho $ and $\| v\| _{\infty
}\leq \rho $. Then
\begin{equation*}
\mu u(t)=\sum_{i=1,2}\gamma_{1i}(t)\Bigl(H_{1i}(\beta_{1i} [u])+L_{1i}(\delta_{1i}
[v])\Bigr)+F_{1}(u,v)(t)
\end{equation*}%
and therefore, since $v(t)\leq \rho ,$ for all $t\in \lbrack 0,1]$,

\begin{eqnarray}
\mu u(t)&\leq &\sum_{i=1,2}\gamma_{1i}(t)h_{1i2}\beta_{1i} [u]+\sum_{i=1,2}\gamma_{1i}(t)l_{1i2}\delta_{1i}
[\rho]+F_{1}(u,v)(t)  \label{dis2} \\
&=&\sum_{i=1,2}\gamma_{1i}(t)h_{1i2}\beta_{1i} [u]+\rho \sum_{i=1,2}\gamma_{1i}(t)l_{1i2}\delta_{1i}
[1]+F_{1}(u,v)(t).  \notag
\end{eqnarray}%
Applying $\beta _{11}$ and $\beta_{12}$ to both sides of \eqref{dis2} gives

\begin{align*}
\mu \beta_{11}[u]&\leq \sum_{i=1,2}\beta_{11}[\gamma_{1i}]h_{1i2}\beta_{1i} [u]+\rho \sum_{i=1,2}\beta_{11}[\gamma_{1i}]l_{1i2}\delta_{1i}
[1]+\beta_{11}[F_1(u,v)],\\
\mu \beta_{12}[u]&\leq \sum_{i=1,2}\beta_{12}[\gamma_{1i}]h_{1i2}\beta_{1i} [u]+\rho \sum_{i=1,2}\beta_{12}[\gamma_{1i}]l_{1i2}\delta_{1i}
[1]+\beta_{12}[F_1(u,v)].
\end{align*}
Thus we have

\begin{align*}
(\mu- h_{112}\beta_{11}[\gamma_{11}]) \beta_{11}[u]-h_{122}\beta_{11}[\gamma_{12}]\beta_{12}[u] &\leq \rho\sum_{i=1,2}\beta_{11}[\gamma_{1i}]l_{1i2}\delta_{1i}
[1]+\beta_{11}[F_1(u,v)],\\
-h_{112}\beta_{12}[\gamma_{11}]{\beta_{11}}[u]+(\mu-h_{122}\beta_{12}[\gamma_{12}]) \beta_{12}[u] &\leq \rho\sum_{i=1,2}\beta_{12}[\gamma_{1i}]l_{1i2}\delta_{1i}
[1]+\beta_{12}[F_1(u,v)],
\end{align*}
that is
\begin{gather}
\begin{aligned}\label{idx1in}
\begin{pmatrix}
\mu- h_{112}\beta_{11}[\gamma_{11}] & -h_{122}\beta_{11}[\gamma_{12}]\\
-h_{112}\beta_{12}[\gamma_{11}] & \mu-h_{122}\beta_{12}[\gamma_{12}]
\end{pmatrix}&
\begin{pmatrix}
\beta_{11}[u]\\
\beta_{12}[u]
\end{pmatrix}\\
\leq&
\begin{pmatrix}
\rho\sum_{i=1,2}\beta_{11}[\gamma_{1i}]l_{1i2}\delta_{1i}
[1]+\beta_{11}[F_1(u,v)]\\
\rho\sum_{i=1,2}\beta_{12}[\gamma_{1i}]l_{1i2}\delta_{1i}
[1]+\beta_{12}[F_1(u,v)]
\end{pmatrix}.
\end{aligned}
\end{gather}
The matrix
$$
\mathcal{M}_{\mu}=
\begin{pmatrix}
\mu-h_{112}\beta_{11}[\gamma_{11}] & -h_{122}\beta_{11}[\gamma_{12}]\\
-h_{112}\beta_{12}[\gamma_{11}] & \mu-h_{122}\beta_{12}[\gamma_{12}]
\end{pmatrix},
$$
satisfies the hypotheses of Lemma \ref{lematrix2}, thus $(\mathcal{M}_{\mu})^{-1}$ is order preserving.
If we apply
$(\mathcal{M}_{\mu})^{-1}$ to both sides of the inequality \eqref{idx1in}  we
obtain
\begin{multline*}
\begin{pmatrix}
\beta_{11}[u]\\
\beta_{12}[u]
\end{pmatrix}
\leq
\frac{1}{\det(\mathcal{M}_{\mu})}
\begin{pmatrix}
\mu-h_{122}\beta_{12}[\gamma_{12}] & h_{122}\beta_{11}[\gamma_{12}]\\
h_{112}\beta_{12}[\gamma_{11}] & \mu-h_{112}\beta_{11}[\gamma_{11}]
\end{pmatrix}
\\
\times
\begin{pmatrix}
\rho\sum_{i=1,2}\beta_{11}[\gamma_{1i}]l_{1i2}\delta_{1i}
[1]+\beta_{11}[F_1(u,v)]\\
\rho\sum_{i=1,2}\beta_{12}[\gamma_{1i}]l_{1i2}\delta_{1i}
[1]+\beta_{12}[F_1(u,v)]
\end{pmatrix},
\end{multline*}
and by Remark \ref{rem1}, we have
\begin{multline*}
\begin{pmatrix}
\beta_{11}[u]\\
\beta_{12}[u]
\end{pmatrix}
\leq
\frac{1}{D_1}
\begin{pmatrix}
1-h_{122}\beta_{12}[\gamma_{12}] & h_{122}\beta_{11}[\gamma_{12}]\\
h_{112}\beta_{12}[\gamma_{11}] & 1-h_{112}\beta_{11}[\gamma_{11}]
\end{pmatrix}
\\ 
\times
\begin{pmatrix}
\rho\sum_{i=1,2}\beta_{11}[\gamma_{1i}]l_{1i2}\delta_{1i}
[1]+\beta_{11}[F_1(u,v)]\\
\rho\sum_{i=1,2}\beta_{12}[\gamma_{1i}]l_{1i2}\delta_{1i}
[1]+\beta_{12}[F_1(u,v)]
\end{pmatrix},
\end{multline*}
that is 
$$
\begin{pmatrix}
\beta_{11}[u]\\
\beta_{12}[u]
\end{pmatrix}
\leq
\begin{pmatrix}
\theta_{11} & \theta_{12}\\
\theta_{13} & \theta_{14}
\end{pmatrix}
\begin{pmatrix}
\rho Q_1+\beta_{11}[F_1(u,v)]\\
\rho S_1+\beta_{12}[F_1(u,v)]
\end{pmatrix}.
$$
Thus
$$
\begin{pmatrix}
\beta_{11}[u]\\
\beta_{12}[u]
\end{pmatrix}
\leq
\begin{pmatrix}
\rho(\theta_{11} Q_1+\theta_{12} S_1)+\theta_{11}\beta_{11}[F_1(u,v)]+\theta_{12}\beta_{12}[F_1(u,v)]\\
\rho(\theta_{13} Q_1+\theta_{14} S_1)+\theta_{13}\beta_{11}[F_1(u,v)]+\theta_{14}\beta_{12}[F_1(u,v)]
\end{pmatrix}.
$$
Substituting into \eqref{dis2} gives

\begin{align*}
\mu u(t) \leq &\rho \Bigl(\gamma_{11}(t)h_{112}(\theta_{11} Q_1+\theta_{12} S_1)+
\gamma_{12}(t)h_{122}(\theta_{13} Q_1+\theta_{14} S_1)+\sum_{i=1,2} \gamma _{1i}(t)l_{1i2}\delta_{1i}
[1] \Bigr) \\
&+\Bigl(\gamma_{11}(t)h_{112}\theta_{11}+
\gamma_{12}(t)h_{122}\theta_{13} \Bigr)\beta_{11}[F_1(u,v)] \\
&+\Bigl(\gamma_{11}(t)h_{112}\theta_{12}+
\gamma_{12}(t)h_{122}\theta_{14} \Bigr)\beta_{12}[F_1(u,v)]\\
&+ F_1(u,v)(t)\\
=&\rho \Bigl(\gamma_{11}(t)h_{112}(\theta_{11} Q_1+\theta_{12} S_1)+
\gamma_{12}(t)h_{122}(\theta_{13} Q_1+\theta_{14} S_1)+\sum_{i=1,2} \gamma _{1i}(t) l_{1i2}\delta_{1i}
[1] \Bigr) \\
&+\Bigl(\gamma_{11}(t)h_{112}\theta_{11}+
\gamma_{12}(t)h_{122}\theta_{13} \Bigr)\int_0^1 \mathcal{K}_{11}(s)g_1(s)f_1(s,u(s),v(s))\,ds \\
&+\Bigl(\gamma_{11}(t)h_{112}\theta_{12}+
\gamma_{12}(t)h_{122}\theta_{14} \Bigr)\int_0^1 \mathcal{K}_{12}(s)g_1(s)f_1(s,u(s),v(s))\,ds+ F_1(u,v)(t).
\end{align*}

Taking the supremum over $[0,1]$ gives
\begin{align*}
\mu {\rho} \leq &\rho \Bigl(\| \gamma _{11}\| _{\infty }h_{112}(\theta_{11} Q_1+\theta_{12} S_1)+
\| \gamma _{12}\| _{\infty }h_{122}(\theta_{13} Q_1+\theta_{14} S_1)+\sum_{i=1,2}\| \gamma _{1i}\| _{\infty }l_{1i2}\delta_{1i}
[1]  \Bigr) \\
&+{\rho} f_1^{0,\rho} \Bigl(\| \gamma _{11}\| _{\infty }h_{112}\theta_{11}+
\| \gamma _{12}\| _{\infty }h_{122}\theta_{13} \Bigr)\int_0^1 \mathcal{K}_{11}(s)g_1(s)\,ds \\
&+{\rho} f_1^{0,\rho}\Bigl(\| \gamma _{11}\| _{\infty }h_{112}\theta_{12}+
\| \gamma _{12}\| _{\infty }h_{122}\theta_{14} \Bigr)\int_0^1 \mathcal{K}_{12}(s)g_1(s)\,ds+{\rho} f_1^{0,\rho} \dfrac{1}{m_1}.
\end{align*}

Using the hypothesis \eqref{eqmestt} we obtain $\mu \rho <\rho .$ This
contradicts the fact that $\mu \geq 1$ and proves the result.
\end{proof}

We give a first Lemma that shows that the index is 0 on a set
$V_{\rho}$.  

\begin{lem}
Assume that
\begin{enumerate}
\item[$(\mathrm{I}^{0}_{\rho})$]  there exist $\rho>0$ such that for every $i=1,2$
\begin{multline}\label{eqMest}
f_{i,(\rho, \rho/c)} \, \Bigl( \bigl(
\dfrac{c_{i1}\|\gamma_{i1}\|h_{i11}}{\underline{D}_i}(1-h_{i21}\beta_{i2}[\gamma_{i2}])
+\dfrac{c_{i2}\|\gamma_{i2}\|h_{i21}}{\underline{D}_i }h_{i11}\beta_{i2}[\gamma_{i1}]\bigr) \int_{a_i}^{b_i}
\mathcal{K}_{i1}(s)g_i(s)\,ds\\
+\bigl(
\dfrac{c_{i1}\|\gamma_{i1}\|h_{i11}}{\underline{D}_i }h_{i21}\beta_{i1}[\gamma_{i2}])
+\dfrac{c_{i2}\|\gamma_{i2}\|h_{i21}}{\underline{D}_i }(1-h_{i11}\beta_{i1}[\gamma_{i1}])\bigr) \int_{a_i}^{b_i}
\mathcal{K}_{i2}(s)g_i(s)\,ds +\dfrac{1}{M_i}\Bigr)>1,
\end{multline}{}
where
\begin{multline*}
f_{1,(\rho,{\rho / c})}=\inf \Bigl\{ \frac{f_1(t,u,v)}{ \rho}:\; (t,u,v)\in [a_1,b_1]\times[\rho,\rho/c]\times[0, \rho/c]\Bigr\},\\
f_{2,(\rho,{\rho / c})}=\inf \Bigl\{ \frac{f_2(t,u,v)}{ \rho}:\; (t,u,v)\in [a_2,b_2]\times[0,\rho/c]\times[\rho, \rho/c]\Bigr\}\\
 \text{and}\ \frac{1}{M_i}=\inf_{t\in
[a_i,b_i]}\int_{a_i}^{b_i} k_i(t,s) g_i(s)\,ds.
\end{multline*}
\end{enumerate}
Then $i_{K}(T,V_{\rho})=0$.
\end{lem}

\begin{proof}
Let $e(t)\equiv 1$ for $t\in [0,1]$. Then $(e,e)\in K$. We prove that
\begin{equation*}
(u,v)\ne T(u,v)+\mu (e,e)\quad\text{for } (u,v)\in \partial
V_{\rho}\quad\text{and } \mu \geq 0.
\end{equation*}
In fact, if this does not happen, there exist $(u,v)\in \partial V_\rho$ and
$\mu \geq 0$ such that $(u,v)=T(u,v)+\mu (e,e)$.
Without loss of generality, we can assume that for all $t\in [a_1,b_1]$ we have
$$
\rho\leq u(t)\leq {\rho/c},\\\ \min u(t)=\rho \\\ \text{and  }\\\ 0\leq v(t)\leq {\rho/c}.
$$
Then, for $t\in [a_1,b_1]$, we obtain
$$
u(t)= \sum_{i=1,2}\gamma_{1i}(t)\Bigl(H_{1i}(\beta_{1i} [u])+L_{1i}(\delta_{1i}[v])\Bigr)+F_{1}(u,v)(t)+ \mu e
$$
and therefore
\begin{equation}\label{diss1}
u(t)\ge \sum_{i=1,2}\gamma_{1i}(t)H_{1i}(\beta_{1i} [u])+F_{1}(u,v)(t)+ \mu e \ge \sum_{i=1,2}\gamma_{1i}(t)h_{1i1}\beta_{1i} [u]+F_{1}(u,v)(t)
+\mu e.
\end{equation}
Applying $\beta_{11}$ and $\beta_{12}$ to both sides of \eqref{diss1} gives
\begin{align*}
\beta_{11}[u] &\ge h_{111}\beta_{11}[\gamma_{11}]\beta_{11}[u]+h_{121}\beta_{11}[\gamma_{12}]\beta_{12}[u]+\beta_{11}[F_1(u,v)]+\mu \beta_{11} [e],\\
\beta_{12}[u] &\ge h_{111}\beta_{12}[\gamma_{11}]\beta_{11}[u]+h_{121}\beta_{12}[\gamma_{12}]\beta_{12}[u]+\beta_{12}[F_1(u,v)]+\mu \beta_{12} [e].
\end{align*}
Thus we have

\begin{align*}
(1-h_{111}\beta_{11}[\gamma_{11}]) \beta_{11}[u] -h_{121}\beta_{11}[\gamma_{12}]\beta_{12}[u] &\ge \beta_{11}[F_1(u,v)]+\mu \beta_{11} [e],\\
-h_{111}\beta_{12}[\gamma_{11}]{\beta_{11}}[u]+(1-h_{121}\beta_{12}[\gamma_{12}]) \beta_{12}[u] &\ge \beta_{12}[F_1(u,v)]+ \mu \beta_{12}[e],
\end{align*}

that is

\begin{multline*}
\begin{pmatrix}
1-h_{111}\beta_{11}[\gamma_{11}] & -h_{121}\beta_{11}[\gamma_{12}]\\
-h_{111}\beta_{12}[\gamma_{11}] & 1-h_{121}\beta_{12}[\gamma_{12}]
\end{pmatrix}
\begin{pmatrix}
\beta_{11}[u]\\
\beta_{12}[u]
\end{pmatrix}
\\
\geq
\begin{pmatrix}
\beta_{11}[F_1(u,v)]+\mu \beta_{11}[e]\\
\beta_{12}[F_1(u,v)]+ \mu \beta_{12}[e]
\end{pmatrix}
\geq
\begin{pmatrix}
\beta_{11}[F_1(u,v)]\\
\beta_{12}[F_1(u,v)]
\end{pmatrix}.
\end{multline*}

The matrix
$$
\underline{\mathcal{M}}_1=
\begin{pmatrix}
1-h_{111}\beta_{11}[\gamma_{11}] & -h_{121}\beta_{11}[\gamma_{12}]\\
-h_{111}\beta_{12}[\gamma_{11}] & 1-h_{121}\beta_{12}[\gamma_{12}]
\end{pmatrix}
$$
satisfies the hypotheses of Lemma \ref{lematrix2}, thus $(\underline{\mathcal{M}}_1)^{-1}$ is order preserving.
If we apply
$(\underline{\mathcal{M}}_1)^{-1}$ to both sides of the last inequality  we
obtain
$$
\begin{pmatrix}
\beta_{11}[u]\\
\beta_{12}[u]
\end{pmatrix}
\geq
\frac{1}{\underline{D}_1}
\begin{pmatrix}
1-h_{121}\beta_{12}[\gamma_{12}] & h_{121}\beta_{11}[\gamma_{12}]\\
h_{111}\beta_{12}[\gamma_{11}] & 1-h_{111}\beta_{11}[\gamma_{11}]
\end{pmatrix}
\begin{pmatrix}
\beta_{11}[F_1(u,v)]\\
\beta_{12}[F_1(u,v)]
\end{pmatrix}
$$
and therefore
\begin{align*}
u(t)\geq &\Bigl(\dfrac{\gamma_{11}(t)}{\underline{D}_1}h_{111}(1-h_{121}\beta_{12}[\gamma_{12}])+
\dfrac{\gamma_{12}(t)}{\underline{D}_1}h_{121}h_{111}\beta_{12}[\gamma_{11}]\Bigr)\\
&\times\int_0^1 \mathcal{K}_{11}(s)g_1(s)f_1(s,u(s),v(s))\,ds\\
&+\Bigl(\dfrac{\gamma_{11}(t)}{\underline{D}_1}h_{111}h_{121}\beta_{11}[\gamma_{12}]+
\dfrac{\gamma_{12}(t)}{\underline{D}_1}(1-h_{111}\beta_{11}[\gamma_{11}])h_{121}\Bigr)\\
& \times \int_0^1 \mathcal{K}_{12}(s)g_1(s)f_1(s,u(s),v(s))\,ds\\
&+  \int_0^1 k_{1}(t,s)g_1(s)f_1(s,u(s),v(s))\,ds + \mu.
\end{align*}
Then we have, for $t\in [a_1,b_1]$,
\begin{align*}
  u(t)\geq &\Bigl(\dfrac{c_{11}\|\gamma_{11}\|}{\underline{D}_1}h_{111}(1-h_{121}\beta_{12}[\gamma_{12}])+
\dfrac{c_{12}\|\gamma_{12}\|}{\underline{D}_1}h_{121}h_{111}\beta_{12}[\gamma_{11}]\Bigr)\\
& \times \int_{a_1}^{b_1}  \mathcal{K}_{11}(s)g_1(s)f_1(s,u(s),v(s))\,ds\\
& +\Bigl(\dfrac{c_{11}\|\gamma_{11}\|}{\underline{D}_1}h_{111}h_{121}\beta_{11}[\gamma_{12}]+
\dfrac{c_{12}\|\gamma_{12}\|}{\underline{D}_1}(1-h_{111}\beta_{11}[\gamma_{11}])h_{121}\Bigr)\\
& \times \int_{a_1}^{b_1}
 \mathcal{K}_{12}(s)g_1(s)f_1(s,u(s),v(s))\,ds +\int_{a_1}^{b_1}k_1(t,s)g_1(s)f_1(s,u(s),v(s))\,ds+{\mu}.
\end{align*}
Taking the minimum over $[a_{1},b_{1}]$ gives
\begin{align*}
\rho= \min_{t\in \lbrack a_{1},b_{1}]}u(t)\geq &{\rho }f_{1,(\rho ,{\rho /c})}\Bigl(\dfrac{c_{11}\|\gamma_{11}\|}{\underline{D}_1}h_{111}(1-h_{121}\beta_{12}[\gamma_{12}])+
\dfrac{c_{12}\|\gamma_{12}\|}{\underline{D}_1}h_{121}h_{111}\beta_{12}[\gamma_{11}]\Bigr)\\
& \times \int_{a_{1}}^{b_{1}}\mathcal{K}_{11}(s)g_{1}(s)\,ds\\
& +{\rho }f_{1,(\rho ,{\rho /c})}\Bigl(\dfrac{c_{11}\|\gamma_{11}\|}{\underline{D}_1}h_{111}h_{121}\beta_{11}[\gamma_{12}]+
\dfrac{c_{12}\|\gamma_{12}\|}{\underline{D}_1}(1-h_{111}\beta_{11}[\gamma_{11}])h_{121}\Bigr)\\
& \times \int_{a_1}^{b_1}
 \mathcal{K}_{12}(s)g_1(s)\,ds +\rho f_{1,(\rho ,{\rho /c})}\frac{1}{M_{1}}+{\mu}.
\end{align*}
Using the hypothesis \eqref{eqMest} we obtain $\rho>\rho +\mu $, a contradiction.
\end{proof}

The following Lemma provides a result of index 0 on $V_{\rho}$ of a different flavour; the idea
is to control the growth of just one nonlinearity $f_i$, at the cost of
having to deal with a larger domain. The proof is omitted as it follows from the previous proof, for details see \cite{gipp-nonlin,gifmpp-cnsns}. We mention that nonlinearities with different growth were studied also in ~\cite{precup1,precup2,ya1}.
\begin{lem}
Assume that
\begin{enumerate}
\item[$(\mathrm{I}^{0}_{\rho})^{\star}$] there exist $\rho>0$ such that for some $i=1,2$
\begin{multline*}
f^*_{i,(0, \rho/c)} \, \Bigl( \bigl(
\dfrac{c_{i1}\|\gamma_{i1}\|h_{i11}}{\underline{D}_i}(1-h_{i21}\beta_{i2}[\gamma_{i2}])
+\dfrac{c_{i2}\|\gamma_{i2}\|h_{i21}}{\underline{D}_i}h_{i11}\beta_{i2}[\gamma_{i1}]\bigr) \int_{a_i}^{b_i}
\mathcal{K}_{i1}(s)g_i(s)\,ds\\
+\bigl(
\dfrac{c_{i1}\|\gamma_{i1}\|h_{i11}}{\underline{D}_i }h_{i21}\beta_{i1}[\gamma_{i2}])
+\dfrac{c_{i2}\|\gamma_{i2}\|h_{i21}}{\underline{D}_i}(1-h_{i11}\beta_{i1}[\gamma_{i1}])\bigr) \int_{a_i}^{b_i}
\mathcal{K}_{i2}(s)g_i(s)\,ds +\dfrac{1}{M_i}\Bigr)>1.\
\end{multline*}{}
\end{enumerate}
where
\begin{equation*}
f^*_{i,(0,{\rho / c})}=\inf \Bigl\{ \frac{f_i(t,u,v)}{ \rho}:\; (t,u,v)\in [a_i,b_i]\times[0,\rho/c]\times[0, \rho/c]\Bigr\}.
\end{equation*}
Then $i_{K}(T,V_{\rho})=0$.
\end{lem}

The above Lemmas can be combined to prove the following Theorem, here we
deal with the existence of at least one, two or three solutions. We stress
that, by expanding the lists in conditions $(S_{5}),(S_{6})$ below, it is
possible, in a similar way as in~\cite{kljdeds}, to state results for four or more positive solutions. We omit
the proof which follows from the properties of fixed point index.

\begin{thm}
\label{thmmsol1} The system \eqref{syst} has at least one positive solution
in $K$ if either of the following conditions hold.

\begin{enumerate}

\item[$(S_{1})$] There exist $\rho _{1},\rho _{2}\in (0,\infty )$ with $\rho
_{1}/c<\rho _{2}$ such that $(\mathrm{I}_{\rho _{1}}^{0})\;\;[\text{or}\;(%
\mathrm{I}_{\rho _{1}}^{0})^{\star }],\;\;(\mathrm{I}_{\rho _{2}}^{1})$ hold.

\item[$(S_{2})$] There exist $\rho _{1},\rho _{2}\in (0,\infty )$ with $\rho
_{1}<\rho _{2}$ such that $(\mathrm{I}_{\rho _{1}}^{1}),\;\;(\mathrm{I}%
_{\rho _{2}}^{0})$ hold.
\end{enumerate}

The system \eqref{syst} has at least two positive solutions in $K$ if one of
the following conditions hold.

\begin{enumerate}

\item[$(S_{3})$] There exist $\rho _{1},\rho _{2},\rho _{3}\in (0,\infty )$
with $\rho _{1}/c<\rho _{2}<\rho _{3}$ such that $(\mathrm{I}_{\rho
_{1}}^{0})\;\;[\text{or}\;(\mathrm{I}_{\rho _{1}}^{0})^{\star }],\;\;(%
\mathrm{I}_{\rho _{2}}^{1})$ $\text{and}\;\;(\mathrm{I}_{\rho _{3}}^{0})$
hold.

\item[$(S_{4})$] There exist $\rho _{1},\rho _{2},\rho _{3}\in (0,\infty )$
with $\rho _{1}<\rho _{2}$ and $\rho _{2}/c<\rho _{3}$ such that $(\mathrm{I}%
_{\rho _{1}}^{1}),\;\;(\mathrm{I}_{\rho _{2}}^{0})$ $\text{and}\;\;(\mathrm{I%
}_{\rho _{3}}^{1})$ hold.
\end{enumerate}

The system \eqref{syst} has at least three positive solutions in $K$ if one
of the following conditions hold.

\begin{enumerate}
\item[$(S_{5})$] There exist $\rho _{1},\rho _{2},\rho _{3},\rho _{4}\in
(0,\infty )$ with $\rho _{1}/c<\rho _{2}<\rho _{3}$ and $\rho _{3}/c<\rho
_{4}$ such that $(\mathrm{I}_{\rho _{1}}^{0})\;\;[\text{or}\;(\mathrm{I}%
_{\rho _{1}}^{0})^{\star }],$ $(\mathrm{I}_{\rho _{2}}^{1}),\;\;(\mathrm{I}%
_{\rho _{3}}^{0})\;\;\text{and}\;\;(\mathrm{I}_{\rho _{4}}^{1})$ hold.

\item[$(S_{6})$] There exist $\rho _{1},\rho _{2},\rho _{3},\rho _{4}\in
(0,\infty )$ with $\rho _{1}<\rho _{2}$ and $\rho _{2}/c<\rho _{3}<\rho _{4}$
such that $(\mathrm{I}_{\rho _{1}}^{1}),\;\;(\mathrm{I}_{\rho
_{2}}^{0}),\;\;(\mathrm{I}_{\rho _{3}}^{1})$ $\text{and}\;\;(\mathrm{I}%
_{\rho _{4}}^{0})$ hold.
\end{enumerate}
\end{thm}

\begin{rem}
If the nonlinearities $f_{1}$ and $f_{2}$ have
some extra positivity properties, for example if the
condition $(S_{1})$ holds and moreover we assume that $f_{1}(t,0,v)>0$ in $%
[a_{1},b_{1}]\times \{0\}\times \lbrack 0,\rho _{2}]$ and $f_{2}(t,u,0)>0$
in $[a_{2},b_{2}]\times \lbrack 0,\rho _{2}]\times \{0\}$, then the solution $%
(u,v)$ of the system \eqref{syst} is such that $\| u\| _{\infty }$ and
$\| v\| _{\infty }$ are strictly positive. 
\end{rem}

\section{An application to coupled systems of BVPs}

We study the existence of positive solutions for the system of second order ODEs
\begin{equation}  \label{eq3.1}
\begin{array}{c}
u^{\prime \prime }(t)+g_{1}(t)f_{1}(t,u(t),v(t))=0,\ t \in (0,1), \\
v^{(4)}(t)=g_{2}(t)f_{2}(t,u(t),v(t)),\ t \in (0,1),%
\end{array}%
\end{equation}
with the nonlocal nonlinear BCs
\begin{gather}
\begin{aligned}   \label{3.bc1}
u(0)=H_{11}(\beta_{11}[u])+L_{11}(\delta_{11}[v]),\ u(1)&=H_{12}(\beta_{12}[u])+L_{12}(\delta_{12}[v]), \\
v(0)=H_{21}(\beta_{21}[v])+L_{21}(\delta_{21}[u]),&\ v^{\prime \prime }(0)=0,\ v(1)=0,\\ 
 \ v^{\prime \prime
}(1)+H_{22}(\beta_{22}[v])+L_{22}&(\delta_{22}[u])=0,%
\end{aligned}
\end{gather}

This differential system can be rewritten in the integral form
\begin{gather*}
\begin{aligned}
u(t)=&(1-t)(H_{11}(\beta_{11}[u])+L_{11}(\delta_{11}[v]))+t(H_{12}(\beta_{12}[u])+L_{12}(\delta_{12}[v]))\\ &+\int_{0}^{1}k_1(t,s)g_1(s)f_1(s,u(s),v(s))\,ds, \\
v(t)=&(1-t)(H_{21}(\beta_{21}[v])+L_{21}(\delta_{21}[u]))+\frac{1}{6}t(1-t^{2})(H_{22}(\beta_{22}[v])+L_{22}(\delta_{22}[u]))\\ &+\int_{0}^{1}k_2(t,s)g_2(s)f_2(s,u(s),v(s))\,ds,%
\end{aligned}
\end{gather*}
where
\begin{equation*}
k_1(t,s)=%
\begin{cases}
s(1-t),\, & s \leq t, \\
t(1-s),\, & s>t,%
\end{cases}
\quad \text{and}\quad k_2(t,s)=
\begin{cases}
\frac{1}{6} s(1-t)(2t-s^{2}-t^{2}),\,  & s\leq t, \\
\frac{1}{6} t(1-s)(2s-t^{2}-s^{2}),\,  & s>t,%
\end{cases}%
\end{equation*}
are non-negative continuous functions on $[0,1]\times[0,1]$.

The intervals $[a_1,b_1]$ and $[a_2,b_2]$ may be chosen arbitrarily in $%
(0,1) $. It is easy to check that
\begin{equation*}
k_1(t,s) \leq s(1-s):=\Phi_1(s), \quad \min_{t \in [a_1,b_1]}k_1(t,s) \geq
c_1 s(1-s),
\end{equation*}
where $c_{1}=\min\{1-b_1,a_1\}$. Furthermore, see~\cite{jw-gi-df}, we have that
\begin{equation*}
k_2(t,s) \leq \Phi_2(s):=
\begin{cases}
\frac{\sqrt{3}}{27}s (1-s^{2})^{\frac{3}{2}}, & \;\; \text{for} \;\; 0\leq
s\leq \frac{1}{2}, \\
\frac{\sqrt{3}}{27}(1-s) s^{\frac{3}{2}}(2-s)^{\frac{3}{2}}, & \;\; \text{for%
} \;\; \frac{1}{2} < s\leq 1,%
\end{cases}%
\end{equation*}
and
\begin{equation*}
k_2(t,s) \geq c_2(t) \Phi_2(s),
\end{equation*}
where
\begin{equation*}
c_{2}(t)=
\begin{cases}
\frac{3\sqrt{3}}{2}t(1-t^{2}), & \;\; \text{for} \;\; t \in [0,1/2], \\
\frac{3\sqrt{3}}{2}t(1-t)(2-t), & \;\; \text{for} \;\; t \in (1/2,1],%
\end{cases}%
\end{equation*}
so that
\begin{equation*}
c_{2}=\min_{t \in [a_2,b_2]}c_{2}(t)>0.
\end{equation*}
The existence of multiple solutions of the system~\eqref{eq3.1}-\eqref{3.bc1} follows from Theorem~\ref{thmmsol1}.

The nonlinearities that occurs in the next example, taken from~\cite{gifmpp-cnsns}, are used to illustrate, under a mathematical point of view, the constants that occur in our theory.
\begin{ex}
Consider the system
\begin{gather}
\begin{aligned}  \label{eq2}
u''+(1/8)(u^3+t^3v^3)+ 2=&0,\ t \in (0,1), \\
v^{(4)}=\sqrt{tu}+13v^2&,\ t \in (0,1), \\
u(0)=H_{11}(u(1/4))+L_{11}(v(1/4)),\ u(1)&=H_{12}(u(3/4))+L_{12}(v(3/4)), \\
v(0)=H_{21}(v(1/3))+L_{21}(u(1/3)),&\ v^{\prime \prime }(0)=0,\ v(1)=0,\\ 
 \ v^{\prime \prime
}(1)+H_{22}(v(2/3))+L_{22}&(u(2/3))=0,%
\end{aligned}
\end{gather}
where the nonlocal conditions are given by the functionals $%
\beta_{ij}[w]=\delta_{ij}[w]=w(\eta_{ij})$ and the functions $H_{ij}$ and $L_{ij}$ 
satisfy the condition
\begin{equation*}
 h_{ij1}w \leq H_{ij}(w)\leq  h_{ij2}w,\,\,\,\,  L_{ij}(w)\leq  l_{ij2}w,
\end{equation*}
with
$$
h_{111}=\frac{1}{6}, h_{112}=\frac{1}{2}, h_{121}=\frac{1}{9}, h_{122}=\frac{1}{3}, h_{211}=\frac{1}{6}, h_{212}=\frac{1}{4}, h_{221}=\frac{1}{2}, h_{222}=\frac{2}{3}$$
$$
l_{112}=\frac{1}{15}, l_{122}=\frac{1}{20}, l_{212}=\frac{1}{20}, l_{222}=\frac{1}{15}.
$$
The functions $H_{ij}$ and $L_{ij}$  can be built in a similar way as in \cite{gipp-nonlin} by choosing, for example,
$$
H_{11}(w)=\left\{
\begin{array}{l}
\frac{1}{2}w,\;\;0 \leq w \leq 1,\\
\frac{1}{6}w+\frac{1}{3},\;\; w \geq 1,
\end{array}
\right.
\quad
L_{11}(w)=\frac{1}{11}\bigl(1+\sin \bigl(w-\frac{\pi}{2} \bigr) \bigr).
$$

The choice $[a_1,b_1]=[a_2,b_2]=[1/4,3/4]$ gives
\begin{equation*}
c_1=1/4,\, c_2=45\sqrt{3}/128,\, c_{11}=c_{12}=c_{21}=1/4,\, c_{22}=45\sqrt{3%
}/128,
\end{equation*}
\begin{equation*}
m_1=8,\, M_1=16,\,m_2=384/5,\, M_2=768/5.
\end{equation*}
We have that 

$$\beta_{11}[\gamma_{11}]=\beta_{12}[\gamma_{12}]=\frac{3}{4}, \beta_{11}[\gamma_{12}]=\beta_{12}[\gamma_{11}]=\frac{1}{4}, \beta_{21}[\gamma_{21}]=\frac{2}{3}, \beta_{21}[\gamma_{22}]=\frac{4}{81}, 
$$
$$\beta_{22}[\gamma_{21}]=\frac{1}{3}, \beta_{22}[\gamma_{22}]=\frac{5}{81}, \delta_{11}[1]=\delta_{12}[1]=\delta_{21}[1]=\delta_{22}[1]=1.$$

Since $\mathcal{K}_{ij}(s) =k_i(\eta_{ij},s)$ we obtain
\begin{equation*}
\int_{0}^{1}\mathcal{K}_{11}(s)\,ds= \int_{0}^{1}\mathcal{K}_{12}(s)\,ds= \frac{3}{32},
\int_{1/4}^{3/4}\mathcal{K}_{11}(s)\,ds= \int_{1/4}^{3/4}\mathcal{K}_{12}(s)\,ds= \frac{1}{16},
\end{equation*}
\begin{equation*}
\int_{0}^{1}\mathcal{K}_{21}(s)\,ds= \int_{0}^{1}\mathcal{K}_{22}(s)\,ds=\frac{11}{972},
\int_{1/4}^{3/4}\mathcal{K}_{11}(s)\,ds= \int_{1/4}^{3/4}\mathcal{K}_{22}(s)\,ds= \frac{3985}{497664}.
\end{equation*}
Then, for $\rho_1=1/8$, $\rho_2=1$ and $\rho_3=11$, we have (the constants that follow have been rounded to 2 decimal places unless exact)\begin{align*}
\inf \Bigl\{ f_1(t,u,v):\; (t,u,v)\in [1/4,3/4]\times[0,1/2]\times[0, 1/2]%
\Bigr\}&=f_1(1/4,0,0)>14.81 \rho_1, \\
\sup \Bigl\{ f_1(t,u,v):\; (t,u,v)\in[0,1]\times[0, 1]\times[0, 1]\Bigr\}%
&=f_1(1,1,1)<2.97 \rho_2, \\
\sup \Bigl\{ f_2(t,u,v):\; (t,u,v)\in[0,1]\times[0,1]\times[0,1]\Bigr\}%
&=f_2(1,1,1)<53.93 \rho_2, \\
\inf \Bigl\{ f_1(t,u,v):\; (t,u,v)\in [1/4,3/4]\times[11,44]\times[0, 44]%
\Bigr\}&=f_1(1/4,11,0)>14.81 \rho_3, \\
\inf \Bigl\{ f_2(t,u,v):\; (t,u,v)\in [1/4,3/4]\times[0,44]\times[11, 44]%
\Bigr\}&=f_2(1/4,0,11)>141.49 \rho_3,
\end{align*}
that is the conditions $(\mathrm{I}^{0}_{\rho_{1}})^{\star}$, $(\mathrm{I}%
^{1}_{\rho_{2}})$ and $(\mathrm{I}^{0}_{\rho_{3}})$ are satisfied; therefore
the system~\eqref{eq2} has at least two positive solutions in $K$.
\end{ex}

\section*{Acknowledgments}
The authors would like to thank Dr. Ing. Antonio Madeo of the Dipartimento di Ingegneria Informatica, Modellistica, Elettronica e Sistemistica - DIMES, Universit\`a della Calabria, for shedding some light on the physical interpretation of the problem.

\end{document}